\documentclass{article}

\usepackage{graphicx} % Required for inserting images
\usepackage{amssymb,amsmath,amsfonts,amsthm,lineno}
\usepackage[mathscr]{euscript}
\usepackage{latexsym}
\usepackage{graphics}
\usepackage{indentfirst}
\usepackage[colorlinks=true,allcolors=blue]{hyperref}
\usepackage[capitalise, noabbrev, nameinlink]{cleveref}
\usepackage{comment}
\usepackage[shortlabels]{enumitem}
\usepackage[english]{babel}
\usepackage{tikz}
\usetikzlibrary{fit,arrows.meta,backgrounds}
\usetikzlibrary{shapes,decorations,graphs,graphs.standard,quotes,backgrounds}
\usetikzlibrary{decorations.markings}
\usepackage{esdiff}
\usepackage{thmtools}
\usepackage{thm-restate}
\usepackage{enumitem}
\makeatletter
\renewcommand{\@fnsymbol}[1]{\number#1}
\makeatother

\setlist[enumerate,1]{label=(\roman*)}

\setlist[enumerate,2]{leftmargin=3em}
\definecolor{azure(colorwheel)}{rgb}{0.0, 0.5, 1.0}
\definecolor{Pink}{RGB}{255, 105, 180}
\definecolor{LSUP}{RGB}{70,29,124}
\definecolor{LSUG}{RGB}{253,208,35}
\definecolor{AUBO}{RGB}{221,85,12}
\definecolor{AUBN}{RGB}{3,36,77}
\definecolor{BGBlue}{RGB}{0,147,210}
\definecolor{NTYTeal}{RGB}{0,165,151}
\definecolor{CampusGreen}{RGB}{78,128,32}
\definecolor{SunkissedYellow}{RGB}{255,192,68}
\definecolor{purple}{RGB}{138
,43,226}
\definecolor{USABlue}{RGB}{0,32,91}
\definecolor{USARed}{RGB}{191,13,62}
\definecolor{USALightBlue}{RGB}{0,169,224}
\definecolor{USAPink}{RGB}{255,128,139}
\definecolor{USATeal}{RGB}{110,206,178}
\definecolor{USAOrange}{RGB}{255,175,2}
\definecolor{USALightGreen}{RGB}{206,213,104}
\definecolor{USADarkGreen}{RGB}{94,103,56}
\definecolor{VandyGold}{HTML}{D8AB4C}
\newtheorem*{theorem*}{Theorem}
\newtheorem{theorem}{Theorem}
\newtheorem{lemma}[theorem]{Lemma}
\newtheorem{pro}[theorem]{Proposition}
\newtheorem{obs}[theorem]{Observation}

\newtheorem{conj}[theorem]{Conjecture}

\crefname{thm}{Theorem}{Theorems}
\crefname{lemma}{Lemma}{Lemmas}

\newcommand{\N}{\mathbb{N}}

\allowdisplaybreaks

\begin{document}

\title{Cones of Theta Graphs are Enumeratively Chromatic-Choosable}
\author{ 
Kingston Barnes \thanks{School of Computing, University of South Alabama, Mobile, AL, USA (kb2436@jagmail.southalabama.edu)}
\and
Vincent Fesenmeier \thanks{McGill-Toolen Catholic High School, Mobile, AL, USA (vincentfesenmeier11@gmail.com)}
\and
Braxton Gaither \thanks{Department of Mathematics and Statistics, University of South Alabama, Mobile, AL, USA (bhg2421@jagmail.southalabama.edu )}
\and
Michael Guthrie \thanks{Department of Mathematics and Statistics, University of South Alabama, Mobile, AL, USA (mag2324@jagmail.southalabama.edu )}
\and
Leonardo Marciaga \thanks{Department of Applied Mathematics, Illinois Institute of Technology, Chicago, IL, USA (lmarciaga@hawk.illinoistech.edu)}
\and
Lucas Miller \thanks{Alabama School of Math and Science, Mobile, AL, USA (lucaskmiller12007@gmail.com)}
\and
Fennec Morrissette \thanks{Department of Biomedical Sciences, University of South Alabama, Mobile, AL, USA (ajm2228@jagmail.southalabama.edu )}
\and
Jeffrey A. Mudrock \thanks{Department of Mathematics and Statistics, University of South Alabama, Mobile, AL, USA (mudrock@southalabama.edu)}
\and
Gavin Nguyen  \thanks{Alabama School of Math and Science, Mobile, AL, USA (Gavin.Nguyen@asms.net)}
\and
Rose Zellner \thanks{Alabama School of Math and Science, Mobile, AL, USA (Sueann.Zellner@asms.net)}
}

\maketitle

\begin{abstract}
Chromatic choosability is a fundamental notion in list coloring. A graph $G$ is \emph{chromatic-choosable} if its chromatic number $\chi(G)$ equals its list chromatic number $\chi_{\ell}(G)$. In 1990, Kostochka and Sidorenko introduced the list color function of a graph $G$, denoted by $P_{\ell}(G,m)$, which is the list-coloring analogue of its chromatic polynomial $P(G,m)$. A graph $G$ is \emph{enumeratively chromatic-choosable} if $P_{\ell}(G,m)=P(G,m)$ for every $m \in \N$. In 2002, Ohba proved that for every graph $G$, there exists an $N \in \N$ such that $K_p \vee G$ is chromatic-choosable whenever $p \geq N$. Allred and Mudrock recently conjectured that an analogous statement holds for enumerative chromatic choosability. In this paper, we prove that $K_1 \vee G$ is enumeratively chromatic-choosable whenever $G$ is a theta graph. Our proof combines classical inequalities, the polynomial method, and ideas from DP (or correspondence) coloring.
 
\medskip



\end{abstract}

\section{Introduction} \label{intro}
In this paper, all graphs or digraphs are nonempty, finite, and simple. Generally speaking, we follow West~\cite{W01} for terminology and notation. The set of natural numbers is $\N=\{1,2,3,\ldots\}$. For $m\in\N$, we write $[m]=\{1,\ldots,m\}$, and we take $[0]$ to be the empty set. When $u$ and $v$ are adjacent in a graph $G$, we write $uv$ (or $vu$) for the edge with endpoints $u$ and $v$. For vertex-disjoint graphs $G$ and $H$, we write $G\vee H$ for the join of $G$ and $H$. The \emph{cone of a graph $G$} is the graph $K_1\vee G$.

\subsection{List Coloring and Chromatic Choosability} \label{basic}

For classical vertex coloring of graphs, we wish to color the vertices of a graph $G$ with up to $m$ colors from $[m]$ so that adjacent vertices receive different colors, a so-called \emph{proper $m$-coloring}.  The \emph{chromatic number} of a graph $G$, denoted by $\chi(G)$, is the smallest $m$ such that $G$ has a proper $m$-coloring.  List coloring is a well-known variation on classical vertex coloring that was introduced independently by Vizing~\cite{V76} and Erd\H{o}s, Rubin, and Taylor~\cite{ET79} in the 1970s.  For list coloring, we associate a \emph{list assignment} $L$ with a graph $G$ such that each vertex $v \in V(G)$ is assigned a list of available colors $L(v)$.  We say that $G$ is \emph{$L$-colorable} if there is a proper coloring $f$ of $G$ such that $f(v) \in L(v)$ for each $v \in V(G)$. We refer to $f$ as a \emph{proper $L$-coloring} of $G$.  A list assignment $L$ for $G$ is called a \emph{$k$-assignment} if $|L(v)|=k$ for each $v \in V(G)$.  A graph $G$ is said to be \emph{$k$-choosable} if it is $L$-colorable whenever $L$ is a $k$-assignment for $G$.  The \emph{list chromatic number} of a graph $G$, denoted by $\chi_\ell(G)$, is the smallest $m$ for which $G$ is $m$-choosable.  For any graph $G$, $\chi(G) \leq \chi_\ell(G)$.
   
A graph $G$ is \emph{chromatic-choosable} if $\chi(G)=\chi_\ell(G)$~\cite{O02}. Determining whether a graph is chromatic-choosable is difficult in general. One of the central conjectures in list coloring, the Edge List Coloring Conjecture, asserts that every line graph of a loopless multigraph is chromatic-choosable  (see~\cite{HC92}). The following result of Ohba~\cite{O02} is important for this paper.

\begin{theorem}[\cite{O02}]
\label{thm: ohbajoin}
For any graph $G$, there exists an $N\in\N$ such that $K_p\vee G$ is chromatic-choosable whenever $p\geq N$.
\end{theorem}

\subsection{Enumeratively Chromatic-Choosable Graphs}

In 1912, Birkhoff~\cite{B12} introduced chromatic polynomials in an effort to make progress on the Four Color Problem.  For $m \in \N$, the \emph{chromatic polynomial} of a graph $G$, denoted by $P(G,m)$, is the number of proper $m$-colorings of $G$.  It can be shown that $P(G,m)$ is a polynomial in $m$ of degree $|V(G)|$ (see~\cite{B12}). For example, for $n \in \N$, $P(K_n,m) = \prod_{i=0}^{n-1} (m-i)$ and $P(C_n,m) = (m-1)^n + (-1)^n (m-1)$.  Moreover, if $T$ is a tree on $n$ vertices, then $P(T,m) = m(m-1)^{n-1}$ (see~\cite{B94, W01}).

The notion of the chromatic polynomial was extended to list coloring in the early 1990s by Kostochka and Sidorenko~\cite{AS90}.  If $L$ is a list assignment for $G$, let $P(G,L)$ denote the number of proper $L$-colorings of $G$.  For $m \in \N$, the \emph{list color function} of $G$, denoted by $P_\ell(G,m)$, is the minimum of $P(G,L)$ over all $m$-assignments $L$ for $G$.  Since an $m$-assignment may assign the same $m$ colors to every vertex of $G$, it follows that $P_\ell(G,m) \leq P(G,m)$ for each $m \in \N$.  In general, the list color function of a graph can differ significantly from its chromatic polynomial for small values of $m$.  For example, for any $n \geq 2$, $P_{\ell}(K_{n,n^n},m) = 0$ and $P(K_{n,n^n},m) > 1$ whenever $m \in \{2, \ldots, n\}$.  On the other hand, in 2023, Dong and Zhang~\cite{DZ22}, improving on earlier results (see~\cite{D92, T09, WQ17}), showed that for any graph $G$, $P_\ell(G,m) = P(G,m)$ whenever $m \geq |E(G)|-1$.

With this and chromatic choosability in mind, it is natural to try to determine which graphs $G$ satisfy $P_\ell(G,m) = P(G,m)$ for every $m \in \N$.  Such graphs are said to be \emph{enumeratively chromatic-choosable}.  Enumerative chromatic choosability was first formally defined in~\cite{KK23}, although it has been studied since the introduction of the list color function in the early 1990s~\cite{AS90}.  Very recently, Dong and Zhang~\cite{DZ26} proved that there are graphs $G$ for which there exists an $m \geq \chi(G)$ such that $P_{\ell}(G,m)=P(G,m)$ and $P_{\ell}(G,m+1) < P(G,m+1)$.  Thus, showing that $P_{\ell}(G,\chi(G)) = P(G,\chi(G))$ does not imply that $G$ is enumeratively chromatic-choosable.  Relatively little is known about which graphs are enumeratively chromatic-choosable.  Our next theorem gives some examples.  For the theorem, recall that a graph is \emph{chordal} if it contains no induced cycle of length greater than 3.

\begin{theorem} \label{thm:examples} \cite{KK23, KM18, KN16, AS90}
The following statements hold.
\begin{enumerate}
    \item\label{examplesa} Chordal graphs are enumeratively chromatic-choosable.
    \item\label{examplesb} Cycles are enumeratively chromatic-choosable.
    \item\label{examplesc} $K_{2,3}$ is enumeratively chromatic-choosable.
    \item\label{examplese} For any graph $G$ and $m \in \N$, if $P_{\ell}(G,m) = P(G,m)$, then $P_{\ell}(K_1 \vee G,m+1) = P(K_1 \vee G,m+1)$.  Consequently, if $G$ is enumeratively chromatic-choosable, then $K_n \vee G$ is enumeratively chromatic-choosable for any $n \in \N$. 
\end{enumerate}
\end{theorem}

A theorem and conjecture served as the motivation for this paper.  First, the enumeratively chromatic-choosable theta graphs were recently characterized.  For $l_1,l_2,l_3 \in \N$, a \emph{theta graph}, denoted $\Theta(l_1,l_2,l_3)$, consists of a pair of end vertices joined by three internally disjoint paths of lengths $l_1$, $l_2$, and $l_3$.

\begin{theorem} \label{thm: characterize} \cite{CL26}
Suppose $G=\Theta(l_1,l_2,l_3)$, where $\min\{l_1,l_2,l_3\} = l_1$, $l_2 \geq 2$, and $l_3 \geq 2$.  Then $G$ is not enumeratively chromatic-choosable if and only if $l_1$, $l_2$, and $l_3$ all have the same parity and $\{l_1,l_2,l_3\} \neq \{2\}$.
\end{theorem}

Second, motivated by Theorem~\ref{thm: ohbajoin}, Allred and Mudrock recently put forth the following conjecture.

\begin{conj} \label{conj:join} \cite{AM25}
For any graph $G$, there is a $p \in \N$ such that $K_p \vee G$ is enumeratively chromatic-choosable.    
\end{conj}

Note that by Theorem~\ref{thm:examples}, if $G$ is an enumeratively chromatic-choosable theta graph, then $K_p \vee G$ is enumeratively chromatic-choosable for each $p \in \N$.  So it is natural to wonder what happens when $G$ is not an enumeratively chromatic-choosable theta graph.  Specifically, is $K_p \vee G$ enumeratively chromatic-choosable for each $p \in \N$ whenever $G$ is a theta graph?  We show that the answer is yes by proving the following theorem.

\begin{theorem} \label{thm: cone}
If $G$ is a theta graph, then $K_1 \vee G$ is enumeratively chromatic-choosable.
\end{theorem}

Our strategy for proving Theorem~\ref{thm: cone} is as follows.  First, suppose $G=\Theta(l_1,l_2,l_3)$, where $l_1 \leq l_2 \leq l_3$ and $l_2 \geq 2$.  If $G$ is enumeratively chromatic-choosable, then the desired result follows from Theorem~\ref{thm:examples}.  Thus, by Theorem~\ref{thm: characterize}, we may assume that $G$ is not enumeratively chromatic-choosable; that is, $l_1$, $l_2$, and $l_3$ all have the same parity and $\{l_1,l_2,l_3\} \neq \{2\}$.  It follows that $G$ is bipartite and $\chi(K_1 \vee G)=3$.  Therefore, it suffices to show that $P_{\ell}(K_1 \vee G,m)=P(K_1 \vee G,m)$ whenever $m \geq 3$.  In Section~\ref{almost}, we show that $P_{\ell}(G,m)=P(G,m)$ whenever $m\geq3$. Our arguments in this section combine bounds involving the DP color function with classical inequalities, including the AM-GM and Karamata inequalities. Together with Theorem~\ref{thm:examples}, this implies that $P_{\ell}(K_1\vee G,m)=P(K_1\vee G,m)$ whenever $m\geq4$.  Finally, in Section~\ref{finish}, we complete the proof of Theorem~\ref{thm: cone} by showing that $P_{\ell}(K_1\vee G,3)=P(K_1\vee G,3)$ using the Alon-Tarsi Theorem and an enumerative form of the Combinatorial Nullstellensatz.

\section{Almost Enumeratively Chromatic-Choosable} \label{almost}

For the remainder of this paper, whenever $G=\Theta(l_1,l_2,l_3)$, we will suppose $l_1 \leq l_2 \leq l_3$ and $l_2 \geq 2$.  In this section we show that when  $l_1$, $l_2$, and $l_3$ all have the same parity and $\{l_1,l_2,l_3\} \neq \{2\}$, $P_{\ell}(G,m)=P(G,m)$ whenever $m \geq 3$.  Intuitively, this shows that such graphs are not enumeratively chromatic-choosable only because $P_{\ell}(G,2) < P(G,2)$.

We begin by using the DP color function.  DP-coloring, introduced by Dvo\v{r}\'{a}k and Postle under the name correspondence coloring~\cite{DP15}, is a generalization of list coloring.  Since we will only use the DP color function, we omit the full definition of DP-coloring; see~\cite{DP15,KM19} for definitions and further background.  For a graph $G$ and $m \in \N$, let $P_{DP}(G,m)$ denote the minimum number of proper DP-colorings over all $m$-fold covers of $G$.  Since every $m$-assignment of $G$ gives rise to an $m$-fold cover of $G$ with the same number of proper colorings, for every graph $G$ and $m \in \N$, $P_{DP}(G,m) \leq P_\ell(G,m) \leq P(G,m)$.  The following lemma involving the DP color function will play an important role in the first part of this section.

\begin{lemma} \cite{AM25} \label{lem: DPconnection}
Suppose $G$ is a graph and $L$ is an $m$-assignment for $G$.  Suppose $uv \in E(G)$, $|L(u)-L(v)| = d \geq 1$, and for any $x \in L(u)$ and $y \in L(v)$ with $x \neq y$, there are at least $C$ proper $L$-colorings of $G$ that color $u$ with $x$ and $v$ with $y$.  Then,
\[P(G,L) \geq P_{DP}(G,m) + Cd.\]
\end{lemma}

Explicit formulas for both the chromatic polynomial and the DP color function are known for theta graphs.  Indeed, when $G=\Theta(l_1,l_2,l_3)$, its chromatic polynomial is given by~\cite{BK23}:
\begin{align*}
&P(G,m)=\\
& \frac{((m-1)^{l_1+1}+(-1)^{l_1+1}(m-1))((m-1)^{l_2+1}+(-1)^{l_2+1}(m-1))((m-1)^{l_3+1}+(-1)^{l_3+1}(m-1))}{(m(m-1))^2} \\
&+ \frac{((m-1)^{l_1}+(-1)^{l_1}(m-1))((m-1)^{l_2}+(-1)^{l_2}(m-1))((m-1)^{l_3}+(-1)^{l_3}(m-1))}{m^2}.
\end{align*}
For the DP color function, we use the following result.

\begin{theorem} \cite{BK23} \label{thm: DPformulas}
Let $G=\Theta(l_1,l_2,l_3)$, where $l_1=\min_{i \in [3]} l_i \geq 1$ and $l_i \geq 2$ for each $i \in \{2,3\}$.  If $l_1$, $l_2$, and $l_3$ all have the same parity, then for $m \geq 3$,
\[
P_{DP}(G,m)=\frac{1}{m}\bigg[(m-1)^{l_1+l_2+l_3}-(m-1)^{l_1}-(m-1)^{l_2}-(m-1)^{l_3}+2(-1)^{l_1+l_2+l_3}\bigg].
\]
\end{theorem}
Finally, the following elementary lemma provides a lower bound on the value of $C$ when applying Lemma~\ref{lem: DPconnection} to $G$.
\begin{lemma} \label{thm: countMath} \cite{CL26}
Suppose $G = \Theta(l_1, l_2, l_3)$ and $L$ is an $m$-assignment for $G$ with $m \geq 3$. Suppose $q$ and $s$ are adjacent vertices in $G$, and suppose $x \in L(q)$, $y \in L(s)$, and $x \neq y$. Then, there are at least
$$(m-1)^{l_1+l_2+l_3-5} (m-2)^2$$
proper $L$-colorings of $G$ that color $q$ with $x$ and $s$ with $y$.
\end{lemma}

Combining these results we prove the following lemma which takes care of many of the possible values of $l_1$, $l_2$, and $l_3$.

\begin{lemma} \label{lem: DPchunk} 
Suppose $G = \Theta(l_1,l_2,l_3)$ where $l_1$, $l_2$, and $l_3$ all have the same parity, $l_1 \leq l_2 \leq l_3$, and $l_2 \geq 2$.  If $l_1 + l_2 \geq 8$, or $l_1 + l_2 = 6$ and $l_3 > l_2$, then $P_{\ell}(G,m) = P(G,m)$ whenever $m \geq 3$. Moreover, when $l_1 + l_2 = 6$ and $l_3 = l_2$, $P_{\ell}(G,m) = P(G,m)$ whenever $m \geq 5$.
\end{lemma}

\begin{proof}
Suppose $L$ is an arbitrary $m$-assignment of $G$ with $m \geq 3$. Note that if we can show $P(G,L) \geq P(G,m)$, then $P_{\ell}(G,m) = P(G,m)$. If $L$ assigns the same list to every vertex of $G$, $P(G,L)=P(G,m)$. So, we may assume there is an edge $qs \in E(G)$ such that $L(q) \neq L(s)$. By Lemma~\ref{lem: DPconnection}, Theorem~\ref{thm: DPformulas}, and Lemma~\ref{thm: countMath}, we have
\begin{align*} P(G,L)
&\geq \frac{1}{m}\bigg[(m-1)^{l_1+l_2+l_3}-(m-1)^{l_1}-(m-1)^{l_2}-(m-1)^{l_3}+2(-1)^{l_1+l_2+l_3}\bigg] \\
&+(m-1)^{l_1+l_2+l_3-5} (m-2)^2.
\end{align*}
We also know
\begin{align*}
&P(G,m) \\
&= \frac{((m-1)^{l_1+1}+(-1)^{l_1+1}(m-1))((m-1)^{l_2+1}+(-1)^{l_2+1}(m-1))((m-1)^{l_3+1}+(-1)^{l_3+1}(m-1))}{(m(m-1))^2} \\
&+ \frac{((m-1)^{l_1}+(-1)^{l_1}(m-1))((m-1)^{l_2}+(-1)^{l_2}(m-1))((m-1)^{l_3}+(-1)^{l_3}(m-1))}{m^2}.
\end{align*}
From this, a calculation shows
$$P(G,L)-P(G,m) \geq (m-1)^{l_1+l_2+l_3-5}(m-2)^2-(m-1)^{l_1}-(m-1)^{l_2}-(m-1)^{l_3}+(-1)^{l_1+1}(m-3).$$ 

Suppose $l_1+l_2 \geq 8$. Notice that 
\begin{align*}
&P(G,L)-P(G,m) \\
&\geq(m-1)^{l_1+l_2+l_3-5}-(m-1)^{l_1}-(m-1)^{l_2}-(m-1)^{l_3}-(m-3) \\
&=(m-1)^{l_1+l_2+l_3-5}\left(1-\frac{(m-1)^5}{(m-1)^{l_2+l_3}}-\frac{(m-1)^5}{(m-1)^{l_1+l_3}}-\frac{(m-1)^5}{(m-1)^{l_1+l_2}}-\frac{(m-1)^6}{(m-1)^{l_1+l_2+l_3}}\right)+2 \\
&\geq (m-1)^{l_1+l_2+l_3-5}(1-4(m-1)^{-3})+2 \\
&\geq 2^{l_1+l_2+l_3-5}\left(\frac{1}{2}\right)+2 > 0.
\end{align*}
Now, suppose $l_1+l_2=6$ and $l_3>l_2$.  Note this means $l_3 \geq l_2+2$ which means $l_1+l_3 \geq 8$ and $l_1+l_2+l_3 \geq 11$. Notice that 
\begin{align*}
&P(G,L)-P(G,m) \\
&\geq(m-1)^{l_1+l_2+l_3-5}-(m-1)^{l_1}-(m-1)^{l_2}-(m-1)^{l_3}-(m-3) \\
&=(m-1)^{l_1+l_2+l_3-5}\left(1-\frac{(m-1)^5}{(m-1)^{l_2+l_3}}-\frac{(m-1)^5}{(m-1)^{l_1+l_3}}-\frac{(m-1)^5}{(m-1)^{l_1+l_2}}-\frac{(m-1)^6}{(m-1)^{l_1+l_2+l_3}}\right)+2 \\
&\geq (m-1)^{l_1+l_2+l_3-5}(1- 2(m-1)^{-3} -(m-1)^{-1}-(m-1)^{-5}) + 2 \\
&\geq 2^{l_1+l_2+l_3-5}\left(\frac{7}{32}\right)+2 > 0.
\end{align*}
Now, suppose $l_1+l_2=6$, $l_3=l_2$, and $m \geq 5$.  Notice that 
\begin{align*}
&P(G,L)-P(G,m) \\
&\geq(m-1)^{l_1+l_2+l_3-5}-(m-1)^{l_1}-(m-1)^{l_2}-(m-1)^{l_3}-(m-3) \\
&=(m-1)^{l_1+l_2+l_3-5}\left(1-\frac{(m-1)^5}{(m-1)^{l_2+l_3}}-\frac{(m-1)^5}{(m-1)^{l_1+l_3}}-\frac{(m-1)^5}{(m-1)^{l_1+l_2}}-\frac{(m-1)^6}{(m-1)^{l_1+l_2+l_3}}\right)+2 \\
&\geq (m-1)^{l_1+l_2+l_3-5}(1 -3(m-1)^{-1}-(m-1)^{-3}) + 2 \\
&\geq 4^{l_1+l_2+l_3-5}\left(\frac{15}{64}\right)+2 > 0.
\end{align*}
Thus, in each case, $P(G,L) \geq P(G,m)$, and hence $P_{\ell}(G,m)=P(G,m)$.
\end{proof}

By Lemma~\ref{lem: DPchunk}, in order to show that $P_{\ell}(G,m)=P(G,m)$ whenever $m \geq 3$, it remains to prove the following statements.
\begin{enumerate}
\item For even $l_3$ satisfying $l_3 \geq 2$ and $m \geq 3$, $P_{\ell}(\Theta(2,2,l_3),m)=P(\Theta(2,2,l_3),m)$.
\item For odd $l_3$ satisfying $l_3 \geq 3$ and $m \geq 3$, $P_{\ell}(\Theta(1,3,l_3),m)=P(\Theta(1,3,l_3),m)$.
\item For $m \in \{3,4\}$, $P_{\ell}(\Theta(1,5,5),m)=P(\Theta(1,5,5),m)$.
\item For $m \in \{3,4\}$, $P_{\ell}(\Theta(2,4,4),m)=P(\Theta(2,4,4),m)$.
\item For $m \in \{3,4\}$, $P_{\ell}(\Theta(3,3,3),m)=P(\Theta(3,3,3),m)$.
\end{enumerate}
Statement~(i) has already been established in the literature.
\begin{theorem} \cite{AM25}
Suppose $G=\Theta(2,2,2k)$ with $k \in \N$.  Then $P_{\ell}(G,m)=P(G,m)$ whenever $m \geq 3$.
\end{theorem}
We now turn our attention to Statement~(ii).  We begin by establishing some notation and conventions that will be used for the remainder of the paper.  When $G=\Theta(l_1,l_2,l_3)$, let $x$ and $y$ denote the common end vertices of the three paths, and let $S_1$, $S_2$, and $S_3$ denote the paths of lengths $l_1$, $l_2$, and $l_3$, respectively.  If $L$ is an $m$-assignment for $G$, $i \in [3]$, and $(c,d) \in L(x) \times L(y)$, let $L_i$ denote the restriction of $L$ to $V(S_i)$, and let $N_i(c,d)$ denote the number of proper $L_i$-colorings of $S_i$ that color $x$ with $c$ and $y$ with $d$.

We are now ready to present two tools that will be used frequently throughout the remainder of the paper.

\begin{lemma} \cite{AM25} \label{lem: sum}
Suppose $G=\Theta(l_1,l_2,l_3)$, and let $L$ be an $m$-assignment for $G$.  Then,
\[
P(G,L)=\sum_{(c,d) \in L(x) \times L(y)} \prod_{i=1}^3 N_i(c,d).
\]
\end{lemma}

\begin{lemma} \cite{CL26} \label{lem: even path simple}
Let $P$ be a path with $k \in \N$ edges whose vertices, in order, are $x_0,x_1,\ldots,x_k$.  Suppose $L$ is an $m$-assignment for $P$ with $m \geq 2$.  For each $(c,d) \in L(x_0) \times L(x_k)$, let $N(c,d)$ denote the number of proper $L$-colorings of $P$ that color $x_0$ with $c$ and $x_k$ with $d$.  Then,
\[
N(c,d) \geq \min \left\{
\frac{(m-1)^k-(-1)^k}{m},
\frac{(m-1)^k-(-1)^k}{m}+(-1)^k
\right\}.
\]
Moreover, there is a partition $\{A,B\}$ of $L(x_0)\times L(x_k)$ such that $|A|=m$, $|B|=m(m-1)$, and
\[
N(c,d)\geq
\begin{cases}
\displaystyle \frac{(m-1)^k-(-1)^k}{m}+(-1)^k,
& \text{if } (c,d)\in A, \\[1em]
\displaystyle \frac{(m-1)^k-(-1)^k}{m},
& \text{if } (c,d)\in B.
\end{cases}
\]
\end{lemma}

We now return to Statement~(ii).  To prove Statement~(ii), it suffices to show that for any $m \geq 3$ and any $m$-assignment $L$ for $G=\Theta(1,3,l_3)$ with $l_3$ odd, $P(G,L) \geq P(G,m)$.  We begin with the case in which $N_2(c,d) \geq ((m-1)^3+1)/m$ for all $(c,d) \in L(x) \times L(y)$ with $c \neq d$.

\begin{lemma} \label{lem: lucasftw}
Suppose $G=\Theta(1,3,l_3)$, where $l_3$ is odd.  Suppose $L$ is an $m$-assignment for $G$ with $m \geq 3$ such that $N_2(c,d) \geq ((m-1)^3+1)/m$ for all $(c,d) \in L(x) \times L(y)$ with $c \neq d$.  Then, $P(G,L)\geq P(G,m)$.
\end{lemma}

\begin{proof}
We must show that $P(G,L) \geq P(G,m)$. By the formula for the chromatic polynomial of a theta graph, we know that
$$P(G,m)=\left(\frac{(m-1)^3+1}{m}\right)\left((m-1)^{l_3+1}+(m-1)\right).$$

Furthermore, by Lemma~\ref{lem: sum},
$$P(G,L) = \sum_{(c,d) \in L(x) \times L(y)} N_1(c,d)N_2(c,d)N_3(c,d).$$
Since $S_1$ is a single edge, $N_1(c,d)=0$ when $c=d$ and $N_1(c,d)=1$ when $c \neq d$. Thus,
$$P(G,L) = \sum_{(c,d) \in L(x) \times L(y),\, c \neq d} N_2(c,d)N_3(c,d).$$
Then, by our hypothesis,
$$P(G,L) \geq \frac{(m-1)^3+1}{m} \sum_{(c,d) \in L(x) \times L(y),\, c \neq d} N_3(c,d).$$

Now, suppose $H$ is the cycle obtained from the path $S_3$ by adding an edge between $x$ and $y$. Notice that
$$P(H,L_3)=\sum_{(c,d) \in L(x) \times L(y),\, c \neq d} N_3(c,d).$$
Since cycles are enumeratively chromatic-choosable by Theorem~\ref{thm:examples} and $H$ is a $(l_3+1)$-cycle,
$$P(H,L_3) \geq P(C_{l_3+1},m)=(m-1)^{l_3+1}+(m-1),$$
where the final equality follows because $l_3$ is odd. Consequently,
$$P(G,L) \geq \frac{(m-1)^3+1}{m}\left((m-1)^{l_3+1}+(m-1)\right)=P(G,m).$$
\end{proof}

In order to handle the remaining list assignments, we prove the following lemma.

\begin{lemma} \label{lem: p4}
Suppose $G=P_4$ and the vertices of $G$, in order, are $x_0,x_1,x_2,x_3$. Suppose $L$ is an $m$-assignment for $G$ with $m \geq 3$, and for each $(c,d) \in L(x_0) \times L(x_3)$, let $N(c,d)$ be the number of proper $L$-colorings of $G$ that color $x_0$ with $c$ and $x_3$ with $d$. Let $I=L(x_1) \cap L(x_2)$. If there exists $(\alpha,\beta) \in L(x_0) \times L(x_3)$ such that $\alpha \neq \beta$ and
$$N(\alpha,\beta)=\frac{(m-1)^3+1}{m}-1=(m-1)(m-2),$$
then the following statements hold.

(i) We have $|I|=m-1$, $L(x_1)=I \cup \{\alpha\}$, and $L(x_2)=I \cup \{\beta\}$.

(ii) For any $(c,d) \in L(x_0) \times L(x_3)-\{(\alpha,\beta)\}$ with $c \neq d$, $N(c,d)>(m-1)(m-2)$.

(iii) If $L(x_0)=L(x_3)$, then $N(\beta,\alpha)=m^2-m+1$, $N(\beta,c) \geq m^2-2m+2$ whenever $c \in L(x_3)-\{\alpha,\beta\}$, and $N(c,\alpha) \geq m^2-2m+2$ whenever $c \in L(x_0)-\{\alpha,\beta\}$.
\end{lemma}

\begin{proof}
First, we prove statement~(i). Let $L'(x_1)=L(x_1)-\{\alpha\}$ and $L'(x_2)=L(x_2)-\{\beta\}$. Then,
$$N(\alpha,\beta)=|L'(x_1)||L'(x_2)|-|L'(x_1)\cap L'(x_2)|.$$
Notice that $|L'(x_i)|\in\{m-1,m\}$ for each $i\in[2]$. If $|I|<m-1$, then
$$(m-1)(m-2)=|L'(x_1)||L'(x_2)|-|L'(x_1)\cap L'(x_2)|\geq (m-1)^2-(m-2),$$
which is a contradiction. Consequently, $|I|\in\{m-1,m\}$.

We claim that $\alpha\in L(x_1)$. Suppose for the sake of contradiction that $\alpha\notin L(x_1)$. Then $|L'(x_1)|=m$. If $\beta\in L(x_2)$, then
$$|L'(x_1)||L'(x_2)|-|L'(x_1)\cap L'(x_2)|\geq m(m-1)-(m-1)>(m-1)(m-2),$$
which is a contradiction. If $\beta\notin L(x_2)$, then
$$|L'(x_1)||L'(x_2)|-|L'(x_1)\cap L'(x_2)|\geq m^2-m>(m-1)(m-2),$$
which is also a contradiction. Thus, $\alpha\in L(x_1)$. A similar argument shows that $\beta\in L(x_2)$.

Now, suppose for the sake of contradiction that $|I|=m$. Then $L(x_1)=L(x_2)$, and hence $|L'(x_1)\cap L'(x_2)|=m-2$. Therefore,
$$|L'(x_1)||L'(x_2)|-|L'(x_1)\cap L'(x_2)|=(m-1)^2-(m-2)>(m-1)(m-2),$$
which is a contradiction. Thus, $|I|=m-1$.

To show that $L(x_1)=I\cup\{\alpha\}$, it remains only to show that $\alpha\notin I$. Suppose, to the contrary, that $\alpha\in I$. Then $|L'(x_1)\cap L'(x_2)|\leq m-2$, and hence
$$|L'(x_1)||L'(x_2)|-|L'(x_1)\cap L'(x_2)|\geq (m-1)^2-(m-2)>(m-1)(m-2),$$
which is a contradiction. Thus, $\alpha\notin I$. A similar argument shows that $\beta\notin I$. Therefore, $L(x_1)=I\cup\{\alpha\}$ and $L(x_2)=I\cup\{\beta\}$, as desired.

We now prove statement~(ii). Suppose $(c,d)\in L(x_0)\times L(x_3)-\{(\alpha,\beta)\}$ with $c\neq d$. Let $L''(x_1)=L(x_1)-\{c\}$ and $L''(x_2)=L(x_2)-\{d\}$. Then
$$N(c,d)=|L''(x_1)||L''(x_2)|-|L''(x_1)\cap L''(x_2)|.$$

Suppose first that $c=\alpha$. Then $|L''(x_1)|=m-1$ and $d\neq\beta$. If $d\in I$, then
$$N(c,d)=(m-1)^2-(m-2)>(m-1)(m-2).$$
If $d\notin I$, then
$$N(c,d)=m(m-1)-(m-1)>(m-1)(m-2).$$
Thus, when $c=\alpha$, $N(c,d)>(m-1)(m-2)$. A similar argument shows that the desired inequality holds when $d=\beta$.

Now, suppose $c\notin I\cup\{\alpha\}$ and $d\notin I\cup\{\beta\}$. Then,
$$N(c,d)=m^2-(m-1)>(m-1)(m-2).$$
Finally, if $|\{c,d\}\cap I|\geq 1$, then
$$N(c,d)\geq (m-1)^2-(m-2)>(m-1)(m-2).$$
This proves statement~(ii).

Finally, we prove statement~(iii). Suppose $L(x_0)=L(x_3)$. If $x_0$ is colored with $\beta$ and $x_3$ is colored with $\alpha$, then
$$N(\beta,\alpha)=|L(x_1)||L(x_2)|-|L(x_1)\cap L(x_2)|=m^2-m+1.$$

Now, suppose $c\in L(x_3)-\{\alpha,\beta\}$. If $x_0$ is colored with $\beta$ and $x_3$ is colored with $c$, we may complete a proper $L$-coloring of $G$ by coloring $x_1$ with an element $d$ of $L(x_1)$ and then coloring $x_2$ with an element of $L(x_2)-\{c,d\}$. We consider two cases. If $c\notin L(x_2)$, then
$$N(\beta,c)=|L(x_1)||L(x_2)|-|L(x_1)\cap L(x_2)|=m^2-m+1\geq m^2-2m+2.$$
If $c\in L(x_2)$, then $c\in I$, and
$$N(\beta,c)=|L(x_1)||L(x_2)-\{c\}|-|L(x_1)\cap(L(x_2)-\{c\})|=m^2-2m+2.$$
A similar argument shows that $N(c,\alpha)\geq m^2-2m+2$ whenever $c\in L(x_0)-\{\alpha,\beta\}$.
\end{proof}

We are now ready to complete the proof of Statement~(ii).

\begin{lemma} \label{lem: finish}
Suppose $G=\Theta(1,3,l_3)$ where $l_3$ is odd. Suppose $L$ is an $m$-assignment for $G$ with $m \geq 3$ such that there exists $(\alpha,\beta) \in L(x) \times L(y)$ with $\alpha \neq \beta$ and
$$N_2(\alpha,\beta)=(m-1)(m-2).$$
Then, $P(G,L)\geq P(G,m)$.
\end{lemma}

\begin{proof}
First, note that the chromatic polynomial of $G$ is given by
$$P(G,m)=m(m-1)\left(\frac{(m-1)^3+1}{m}\right)\left(\frac{(m-1)^{l_3}+1}{m}\right).$$
By Lemma~\ref{lem: sum},
$$P(G,L)=\sum_{(c,d) \in L(x) \times L(y)}N_1(c,d)N_2(c,d)N_3(c,d).$$
Since $S_1$ is a single edge, $N_1(c,d)=0$ when $c=d$ and $N_1(c,d)=1$ when $c\neq d$. Thus,
$$P(G,L)=\sum_{(c,d) \in L(x) \times L(y),\,c\neq d}N_2(c,d)N_3(c,d).$$

We first show the desired result when $L(x)\neq L(y)$. In this case, the sum has at least $m(m-1)+1$ terms. Let
$$a_m=\frac{(m-1)^3+1}{m}\qquad\text{and}\qquad b_m=\frac{(m-1)^{l_3}+1}{m}.$$
Then, by the AM-GM inequality, Lemma~\ref{lem: even path simple}, and Lemma~\ref{lem: p4},
$$P(G,L)\geq \left(m(m-1)+1\right)\left((a_m-1)a_m^{m^2-m}(b_m-1)^m b_m^{m^2-2m+1}\right)^{\frac{1}{m^2-m+1}}.$$
Furthermore, in the special case that $l_3=3$, we have
$$P(G,L)\geq \left(m(m-1)+1\right)\left((a_m-1)^2a_m^{2m^2-2m}\right)^{\frac{1}{m^2-m+1}}.$$

We first consider the case $l_3=3$. Since $m\geq 3$,
\begin{align*}
1+\frac{2}{(m-1)(m-2)}+\frac{1}{(m-1)^2(m-2)^2}
&\leq 1+1+\frac{1}{2}\\
&\leq 1+\frac{m^2-m+1}{m(m-1)}
+\frac{m^2-m+1}{2(m^2-m)}\\
&\leq \left(1+\frac{1}{m(m-1)}\right)^{m^2-m+1}.
\end{align*}
Thus,
$$\left(\frac{a_m}{a_m-1}\right)^2
=\left(1+\frac{1}{(m-1)(m-2)}\right)^2
\leq \left(1+\frac{1}{m(m-1)}\right)^{m^2-m+1}.$$
After rearranging,
$$(m^2-m+1)\left((a_m-1)^2a_m^{2(m^2-m)}\right)^{\frac{1}{m^2-m+1}}
\geq m(m-1)a_m^2,$$
which implies $P(G,L)\geq P(G,m)$.

Now, suppose $l_3\geq 5$. Since $m\geq 3$, one checks that
$$\ln(2.5)\geq \frac{1}{(m-1)(m-2)}
+\frac{m^2}{(m-1)^{l_3}+1-m}
=\frac{1}{a_m-1}+\frac{m}{b_m-1}.$$
Consequently,
$$2.5\geq \exp\left(\frac{1}{a_m-1}+\frac{m}{b_m-1}\right).$$
Using $1+z\leq e^z$, we obtain
$$2.5\geq
\left(1+\frac{1}{a_m-1}\right)
\left(1+\frac{1}{b_m-1}\right)^m
=\left(\frac{a_m}{a_m-1}\right)
\left(\frac{b_m}{b_m-1}\right)^m.$$
Also,
$$\left(1+\frac{1}{m(m-1)}\right)^{m^2-m+1}
\geq 1+\frac{m^2-m+1}{m(m-1)}
+\frac{m^2-m+1}{2(m^2-m)}
\geq 2.5.$$
Therefore,
$$\left(1+\frac{1}{m(m-1)}\right)^{m^2-m+1}
\geq
\left(\frac{a_m}{a_m-1}\right)
\left(\frac{b_m}{b_m-1}\right)^m.$$
After rearranging,
$$\left(m(m-1)+1\right)
\left((a_m-1)a_m^{m^2-m}(b_m-1)^m b_m^{m^2-2m+1}\right)^{\frac{1}{m^2-m+1}}
\geq m(m-1)a_mb_m=P(G,m).$$
Thus, $P(G,L)\geq P(G,m)$ when $L(x)\neq L(y)$.

Now, suppose $L(x)=L(y)$. In this case,
$$P(G,L)=\sum_{(c,d) \in L(x) \times L(y),\,c\neq d}N_2(c,d)N_3(c,d),$$
and this sum has exactly $m(m-1)$ terms. Notice that
$$a_m=\frac{(m-1)^3+1}{m}=m^2-3m+3.$$
Then, by the AM-GM inequality and Lemma~\ref{lem: p4},
\begin{align*}
P(G,L)
&\geq m(m-1)\Big((a_m-1)(a_m+2m-2)(a_m+m-1)^{2m-4}\\
&\hspace{3cm}\cdot a_m^{m^2-3m+2}(b_m-1)^m b_m^{m^2-2m}\Big)^{\frac{1}{m(m-1)}}.
\end{align*}

Since $l_3\geq 3$, we have $b_m\geq a_m$. Consequently,
$$\left(1+\frac{1}{b_m-1}\right)^m
\leq \left(1+\frac{1}{a_m-1}\right)^m.$$
Also, for $m\geq 3$,
$$\left(1-\frac{1}{a_m}\right)\left(1+\frac{2m-2}{a_m}\right)
=1+\frac{2m-3}{a_m}-\frac{2m-2}{a_m^2}\geq 1,$$
and
$$\frac{m-1}{a_m}\geq \frac{1}{a_m-1}.$$

When $m\geq 4$, we have $2m-4\geq m$, and hence
$$\left(1+\frac{m-1}{a_m}\right)^{2m-4}
\geq \left(1+\frac{1}{a_m-1}\right)^{2m-4}
\geq \left(1+\frac{1}{a_m-1}\right)^m
\geq \left(1+\frac{1}{b_m-1}\right)^m.$$
Therefore, when $m\geq 4$,
$$\left(1-\frac{1}{a_m}\right)
\left(1+\frac{2m-2}{a_m}\right)
\left(1+\frac{m-1}{a_m}\right)^{2m-4}
\geq \left(1+\frac{1}{b_m-1}\right)^m.$$

We claim the same inequality holds when $m=3$. In this case, $a_3=3$ and $b_3\geq 3$, so
\begin{align*}
\left(1-\frac{1}{a_m}\right)
\left(1+\frac{2m-2}{a_m}\right)
\left(1+\frac{m-1}{a_m}\right)^{2m-4}
&=\left(\frac{2}{3}\right)\left(\frac{7}{3}\right)\left(\frac{5}{3}\right)^2\\
&=\frac{350}{81}\\
&\geq \frac{27}{8}\\
&\geq \left(1+\frac{1}{b_m-1}\right)^3.
\end{align*}
Thus, for all $m\geq 3$,
$$\left(1-\frac{1}{a_m}\right)
\left(1+\frac{2m-2}{a_m}\right)
\left(1+\frac{m-1}{a_m}\right)^{2m-4}
\geq \left(1+\frac{1}{b_m-1}\right)^m.$$
After rearranging,
$$(a_m-1)(a_m+2m-2)(a_m+m-1)^{2m-4}(b_m-1)^m
\geq a_m^{2m-2}b_m^m.$$
It follows that
$$(a_m-1)(a_m+2m-2)(a_m+m-1)^{2m-4}
a_m^{m^2-3m+2}(b_m-1)^m b_m^{m^2-2m}
\geq a_m^{m^2-m}b_m^{m^2-m}.$$
Therefore,
$$P(G,L)\geq m(m-1)a_mb_m=P(G,m).$$
This completes the proof.
\end{proof}

By Lemma~\ref{lem: even path simple}, either the hypothesis of Lemma~\ref{lem: lucasftw} holds or there exists $(\alpha,\beta)\in L(x)\times L(y)$ with $\alpha\neq\beta$ such that $N_2(\alpha,\beta)=(m-1)(m-2)$. Thus, Lemmas~\ref{lem: lucasftw} and~\ref{lem: finish} together establish Statement~(ii). It remains to prove Statements~(iii), (iv), and (v), which we handle in the next three lemmas. We begin with Statement~(iii).

\begin{lemma} \label{lem: 155amgm}
Suppose $G= \Theta(1,5,5)$. Then, $P_{\ell}(G,m)=P(G,m)$ whenever $m \in \{3,4\}$.
\end{lemma}
\begin{proof}
It is easy to calculate that $P(G,3)=726$ and $P(G,4)=44652$.  Suppose $L$ is an arbitrary $m$-assignment of $G$ where $m \in \{3,4\}$. We must show $P(G,L) \geq 726$ if $m=3$ and $P(G,L) \geq 44652$ if $m=4$.  By Lemma~\ref{lem: sum} and the fact that $N_1(c,d) = 0$ when $c=d$ and $N_1(c,d)=1$ when $c \neq d$, we know that $$P(G,L)= {\sum_{(c,d) \in L(x) \times L(y), c \neq d} N_2(c,d)N_3(c,d)}.$$ 
We first show the desired result in the case $L(x) \neq L(y)$. Since $L(x) \neq L(y)$, we have $|L(x)\cap L(y)|\leq m-1$. Thus, there are at least $m^2-m+1$ pairs $(c,d)\in L(x)\times L(y)$ such that $c\neq d$. Since each term in the sum is positive, we may restrict the sum to any $m^2-m+1$ such pairs. Let $a_m=((m-1)^5+1)/m$, and note that $a_3=11$ and $a_4=61$. By Lemma~\ref{lem: even path simple}, among the resulting $2(m^2-m+1)$ factors $N_2(c,d)$ and $N_3(c,d)$, at most $2m$ are only guaranteed to be at least $a_m-1$, while all remaining factors are at least $a_m$. Thus, by the AM-GM inequality,
$$P(G,L)\geq (m^2-m+1)\left(a_m^{2m^2-4m+2}(a_m-1)^{2m}\right)^{1/(m^2-m+1)}.$$ 
When $m=3$,
$$ P(G,L)\geq 7(11^8\cdot 10^6)^{1/7}>780>726=P(G,3).$$
When $m=4$,
$$ P(G,L)\geq 13(61^{18}\cdot 60^8)^{1/13}>47883>44652=P(G,4).$$
So, the desired result holds when $L(x)\neq L(y)$.

Now suppose that $L(x) = L(y)$.  First, note that the AM-GM inequality gives
\begin{align*}
P(G,L) &\geq(m^2-m) \left( \prod_{(c,d) \in L(x) \times L(y), c \neq d} N_2(c,d)N_3(c,d) \right)^{1/(m^2-m)} \\
&= (m^2-m) \left( \prod_{(c,d) \in L(x) \times L(y), c \neq d} N_2(c,d) \right)^{1/(m^2-m)}\left( \prod_{(c,d) \in L(x) \times L(y), c \neq d} N_3(c,d) \right)^{1/(m^2-m)}.    
\end{align*}
We claim that
$$\prod_{(c,d)\in L(x)\times L(y),\,c\neq d}N_i(c,d)\geq a_m^{m^2-m}$$
for each $i\in\{2,3\}$. This will imply
$$P(G,L)\geq (m^2-m)a_m^2=P(G,m).$$

We will prove our claimed lower bound on $\prod_{(c,d)\in L(x)\times L(y),\,c\neq d}N_2(c,d)$ since the same argument works when $i=3$. If $L_2$ is constant, then
$$\prod_{(c,d)\in L(x)\times L(y),\,c\neq d}N_2(c,d)=a_m^{m^2-m}.$$
So, suppose $L_2$ is not constant. Let $H$ be the graph obtained from $S_2$ by adding the edge $xy$. Then $H=C_6$ and
$$P(H,L_2)=\sum_{(c,d)\in L(x)\times L(y),\,c\neq d}N_2(c,d).$$
Since $H$ is connected and $L_2$ is nonconstant, there are adjacent vertices $q,s\in V(H)$ such that $L(q)\neq L(s)$. Thus, $|L(q)-L(s)|\geq 1$. For any $a\in L(q)$ and $b\in L(s)$ with $a\neq b$, after coloring $q$ with $a$ and $s$ with $b$, there are at least $(m-1)^3(m-2)$ ways to greedily complete the coloring around $H$. Hence, by Lemma~\ref{lem: DPconnection} and the formula $P_{DP}(C_6,m)=(m-1)^6-1$ (see e.g.~\cite{KM19}),
$$P(H,L_2)\geq P_{DP}(H,3)+8=63+8=71$$
when $m=3$, and
$$P(H,L_2)\geq P_{DP}(H,4)+54=728+54=782$$
when $m=4$.

By Lemma~\ref{lem: even path simple}, all $m(m-1)$ terms in the sum giving $P(H,L_2)$ are at least $a_m-1$, at most $m$ of the terms can equal $a_m-1$, and all remaining terms are at least $a_m$. Since increasing any term increases the product, to minimize the product it suffices to assume that the sum attains its lower bound. Subject to this fixed sum and the preceding lower bounds, the product is minimized by taking $m$ terms equal to $a_m-1$, all but one of the remaining terms equal to $a_m$, and concentrating the remaining excess in the last term. Indeed, this sequence majorizes every other admissible sequence with the same sum. Thus, Karamata's inequality (see~\cite[Page 11]{MOA11}), applied to the convex function $\phi(x)=-\log(x)$, shows that this sequence minimizes the product.
Thus, when $m=3$,
$$\prod_{(c,d)\in L(x)\times L(y),\,c\neq d}N_2(c,d)\geq 10^3\cdot 11^2\cdot(71-3(10)-2(11))\geq 11^6.$$
Similarly, when $m=4$,
$$\prod_{(c,d)\in L(x)\times L(y),\,c\neq d}N_2(c,d)\geq 60^4\cdot 61^7\cdot(782-4(60)-7(61))\geq 61^{12}.$$
So, $P(G,L) \geq P(G,m)$.
\end{proof}

We now prove Statement~(iv).

\begin{lemma} \label{lem: 244amgm}
Suppose $G= \Theta(2,4,4)$. Then, $P_{\ell}(G,m)=P(G,m)$ whenever $m \in \{3,4\}$.
\end{lemma}
\begin{proof}
It is easy to calculate that $P(G,3)=366$ and $P(G,4)=14892$. Suppose $L$ is an arbitrary $m$-assignment for $G$, where $m\in\{3,4\}$. We must show that $P(G,L)\geq 366$ when $m=3$ and $P(G,L)\geq 14892$ when $m=4$. Recall that the internal vertex of $S_1$ is $x_1$. We consider separately the cases in which $L(x)=L(x_1)=L(y)$ and in which these three lists are not all equal.

First, suppose $L(x)$, $L(x_1)$, and $L(y)$ are not all equal. Let $\beta=|L(x)\cap L(x_1)\cap L(y)|$, $\gamma=|L(x_1)\cap(L(x)-L(y))|$, and $\mu=|L(x_1)\cap(L(y)-L(x))|$. Notice that $\beta\in\{0,1,\ldots,m-1\}$ and $\beta+\gamma+\mu\leq |L(x_1)|=m$. We have $N_1(c,d)\geq m-2$ for each $(c,d)\in L(x)\times L(y)$. Furthermore, in order for $N_1(c,d)=m-2$, we must have $c\neq d$ and $c,d\in L(x_1)$. Consequently, the number of pairs $(c,d)\in L(x)\times L(y)$ satisfying $N_1(c,d)=m-2$ is
$$(\gamma+\beta)(\mu+\beta)-\beta.$$
Since $\beta+\gamma+\mu\leq m$, it is straightforward to check that the maximum value of this expression is $4$ when $m=3$ and $9$ when $m=4$. By Lemma~\ref{lem: sum},
$$P(G,L)=\sum_{(c,d)\in L(x)\times L(y)}N_1(c,d)N_2(c,d)N_3(c,d).$$
Suppose $m=3$. Then, among the nine terms in the sum above, the factor $N_1(c,d)$ is equal to $1$ at most four times and is at least $2$ at least five times. Thus, by Lemma~\ref{lem: even path simple} and the AM-GM inequality,
$$P(G,L)\geq 9\left(2^5\cdot 6^3\cdot 5^6\cdot 6^3\cdot 5^6\right)^{1/9}>366=P(G,3).$$

Now, suppose $m=4$. Then, among the sixteen terms in the sum above, the factor $N_1(c,d)$ is equal to $2$ at most nine times and is at least $3$ at least seven times. Thus, by Lemma~\ref{lem: even path simple} and the AM-GM inequality,
$$P(G,L)\geq 16\left(2^9\cdot 3^7\cdot 21^4\cdot 20^{12}\cdot 21^4\cdot 20^{12}\right)^{1/16}>14892=P(G,4).$$

Now, suppose $L(x)=L(x_1)=L(y)$. Note that if $L$ is a constant $m$-assignment, then the desired result follows immediately. So, we may assume $L$ is nonconstant. Let $H=S_2\cup S_3$, and let $L'$ be the restriction of $L$ to $V(H)$. Since $L(x)=L(x_1)=L(y)$ and $L$ is nonconstant, $L'$ is a nonconstant $m$-assignment for the $8$-cycle $H$. Thus, there are adjacent vertices $u,v\in V(H)$ such that $L'(u)\neq L'(v)$. Suppose $c_1\in L'(u)$ and $c_2\in L'(v)$ with $c_1\neq c_2$. The number of proper $L'$-colorings of $H$ that color $u$ with $c_1$ and $v$ with $c_2$ is the same as the number of proper $L'$-colorings of $H-uv$ that color $u$ with $c_1$ and $v$ with $c_2$. Since $H-uv$ is a path with $7$ edges, Lemma~\ref{lem: even path simple} implies that this number is at least $42$ when $m=3$ and at least $546$ when $m=4$.

Thus, by Lemma~\ref{lem: DPconnection} and the formula $P_{DP}(C_8,m)=(m-1)^8-1$ (see~\cite{KM19}), when $m=3$ we have
$$P(H,L')\geq P_{DP}(C_8,3)+42=255+42=297,$$
and when $m=4$ we have
$$P(H,L')\geq P_{DP}(C_8,4)+546=6560+546=7106.$$
By Lemma ~\ref{lem: sum} we know that
\[P(G,L) = \sum_{(c,d) \in L(x) \times L(y)} N_1(c,d)N_2(c,d)N_3(c,d).\]

Suppose $m=3$. Name the elements of $L(x)\times L(y)$ so that $L(x)\times L(y)=\{\boldsymbol{d}_1,\ldots,\boldsymbol{d}_9\}$. Since $L(x)=L(x_1)=L(y)$, there are exactly three pairs $(c,d)\in L(x)\times L(y)$ such that $N_1(c,d)=2$, while $N_1(c,d)=1$ for all remaining pairs. Thus, we may assume $N_1(\boldsymbol{d}_i)=2$ for each $i\in[3]$. For each $i\in[9]$, let $d_i=N_2(\boldsymbol{d}_i)N_3(\boldsymbol{d}_i)$. Then,
$$P(H,L')=\sum_{i=1}^9d_i\geq297.$$
By Lemma~\ref{lem: even path simple}, $N_2(\boldsymbol{d}_i)\geq5$ and $N_3(\boldsymbol{d}_i)\geq5$ for each $i\in[9]$, so $d_i\geq25$. Therefore,
$$P(G,L)\geq2\sum_{i=1}^3d_i+\sum_{i=4}^9d_i=(d_1+d_2+d_3)+\sum_{i=1}^9d_i\geq3(25)+297>366=P(G,3).$$

Now, suppose $m=4$. Name the elements of $L(x)\times L(y)$ so that $L(x)\times L(y)=\{\boldsymbol{d}_1,\ldots,\boldsymbol{d}_{16}\}$. Since $L(x)=L(x_1)=L(y)$, there are exactly four pairs $(c,d)\in L(x)\times L(y)$ such that $N_1(c,d)=3$, while $N_1(c,d)=2$ for all remaining pairs. Thus, we may assume $N_1(\boldsymbol{d}_i)=3$ for each $i\in[4]$. For each $i\in[16]$, let $d_i=N_2(\boldsymbol{d}_i)N_3(\boldsymbol{d}_i)$. Then,
$$P(H,L')=\sum_{i=1}^{16}d_i\geq7106.$$
By Lemma~\ref{lem: even path simple}, $d_i\geq20^2=400$ for each $i\in[16]$. Therefore,
$$P(G,L)\geq3\sum_{i=1}^4d_i+2\sum_{i=5}^{16}d_i=(d_1+d_2+d_3+d_4)+2\sum_{i=1}^{16}d_i\geq4(400)+2(7106)>14892=P(G,4).$$
\end{proof}

Finally, we prove Statement~(v).

\begin{lemma} \label{lem: 3s}
Suppose $G= \Theta(3,3,3)$. Then, $P_{\ell}(G,m)=P(G,m)$ whenever $m \in \{3,4\}$.
\end{lemma}
\begin{proof}
It is easy to calculate that $P(G,3)=186$ and $P(G,4)=4980$. Suppose $L$ is an arbitrary $m$-assignment of $G$, where $m \in \{3,4\}$. We must show $P(G,L) \geq 186$ if $m=3$ and $P(G,L) \geq 4980$ if $m=4$.

If $L$ is a constant list assignment, the result immediately follows. So, we may assume $L$ is nonconstant. By symmetry, we may without loss of generality assume $L_1$ is nonconstant. Recall that the internal vertices of $S_1$ are $x_1$ and $x_2$. We have
$$P(S_1,L_1)=\sum_{(c,d) \in L(x) \times L(y)}N_1(c,d).$$
By Lemma~\ref{lem: even path simple}, when $m=3$ each term of this sum is at least $2$, with at most three terms equal to $2$, and when $m=4$ each term is at least $6$, with at most four terms equal to $6$.

Since $L_1$ is nonconstant, there must be an edge $e$ of $S_1$ such that the lists corresponding to the endpoints of $e$ are different. Thus, there are at least $m(m-1)+1$ ways to properly color $e$. Each such coloring can be greedily extended to a proper $L_1$-coloring of $S_1$ in at least $(m-1)^2$ ways. Therefore, $P(S_1,L_1)\geq 7\cdot 4=28$ when $m=3$, and $P(S_1,L_1)\geq 13\cdot 9=117$ when $m=4$.

First, suppose $L(x_1)\neq L(x_2)$. When $m=3$, at most one term of $\sum_{(c,d)\in L(x)\times L(y)}N_1(c,d)$ can equal $2$, since this can only happen for a pair $(q,s)\in L(x)\times L(y)$ satisfying $L(x_1)-\{q\}=L(x_2)-\{s\}$. Similarly, when $m=4$, at most one term can equal $6$. Using Karamata's inequality (see~\cite[Page 11]{MOA11}) applied to the convex function $\phi(x)=-\log(x)$, together with the lower bounds on $P(S_1,L_1)$ above, we obtain
$$\prod_{(c,d)\in L(x)\times L(y)}N_1(c,d)\geq 2\cdot 3^7\cdot(28-2-7(3))=2\cdot 3^7\cdot 5$$
when $m=3$, and
$$\prod_{(c,d)\in L(x)\times L(y)}N_1(c,d)\geq 6\cdot 7^{14}\cdot(117-6-14(7))=6\cdot 7^{14}\cdot 13$$
when $m=4$.

Now, suppose $L(x_1)=L(x_2)$. Since $L_1$ is nonconstant, by reversing $S_1$ if necessary, we may assume $L(x)\neq L(x_1)$. Choose $z\in L(x)-L(x_1)$. When $m=3$, at most two terms of $\sum_{(c,d)\in L(x)\times L(y)}N_1(c,d)$ can equal $2$, since this can only happen when $c=d\in L(x_1)$. Moreover, for each $d\in L(y)$, $N_1(z,d)\geq4$ since $z\notin L(x_1)$. Thus, among the nine terms in the sum, at most two are equal to $2$, three terms are at least $4$, and all remaining terms are at least $3$. Consequently,
$$\prod_{(c,d)\in L(x)\times L(y)}N_1(c,d)\geq 2^2\cdot 4^3\cdot 3^4.$$

Similarly, when $m=4$, at most three terms can equal $6$, and for each $d\in L(y)$, $N_1(z,d)\geq9$. Thus,
$$\prod_{(c,d)\in L(x)\times L(y)}N_1(c,d)\geq 6^3\cdot 9^4\cdot 7^9.$$

Since $2\cdot 3^7\cdot 5>2^2\cdot 4^3\cdot 3^4$ and $6^3\cdot 9^4\cdot 7^9>6\cdot 7^{14}\cdot 13$, we have
$$\prod_{(c,d)\in L(x)\times L(y)}N_1(c,d)\geq 2^2\cdot 4^3\cdot 3^4$$
when $m=3$, and
$$\prod_{(c,d)\in L(x)\times L(y)}N_1(c,d)\geq 6\cdot 7^{14}\cdot 13$$
when $m=4$.

By Lemma~\ref{lem: sum},
$$P(G,L)=\sum_{(c,d)\in L(x)\times L(y)}N_1(c,d)N_2(c,d)N_3(c,d).$$

Suppose $m=3$. By Lemma~\ref{lem: even path simple}, for each $i\in\{2,3\}$, at most three of the factors $N_i(c,d)$ can equal $2$, while all remaining factors are at least $3$. Thus,
$$\prod_{(c,d)\in L(x)\times L(y)}N_i(c,d)\geq 2^3\cdot 3^6.$$
Therefore, by the AM-GM inequality,
$$P(G,L)\geq 9\left(2^2\cdot 4^3\cdot 3^4\cdot 2^6\cdot 3^{12}\right)^{1/9}>186=P(G,3).$$

Now, suppose $m=4$. By Lemma~\ref{lem: even path simple}, for each $i\in\{2,3\}$, at most four of the factors $N_i(c,d)$ can equal $6$, while all remaining factors are at least $7$. Thus,
$$\prod_{(c,d)\in L(x)\times L(y)}N_i(c,d)\geq 6^4\cdot 7^{12}.$$
Therefore, by the AM-GM inequality,
$$P(G,L)\geq 16\left(6\cdot 7^{14}\cdot 13\cdot 6^8\cdot 7^{24}\right)^{1/16}>4980=P(G,4).$$
\end{proof}

\section{Finishing the Proof} \label{finish}

In this section, we complete the proof of Theorem~\ref{thm: cone}. As noted in Section~\ref{intro}, if $G = \Theta(l_1,l_2,l_3)$ is enumeratively chromatic-choosable, then the desired result follows from Theorem~\ref{thm:examples}. Thus, by Theorem~\ref{thm: characterize}, we may assume throughout this section that $G=\Theta(l_1,l_2,l_3)$ is not enumeratively chromatic-choosable; that is, $l_1$, $l_2$, and $l_3$ all have the same parity and $\{l_1,l_2,l_3\}\neq\{2\}$. It follows that $G$ is bipartite. Throughout this section, we retain the notation and conventions introduced in Section~\ref{almost}. Let $X$ and $Y$ denote the partite sets of $G$. Without loss of generality, assume $|X|\geq |Y|$. It is easy to see that $|V(G)|=l_1+l_2+l_3-1$ and $|Y|\geq (|V(G)|-1)/2$.

Throughout this section, let $H=K_1\vee G$, and let $w$ denote the universal vertex used to construct $H$. It suffices to prove the following theorem.

\begin{theorem} \label{thm: conetheta}
Suppose $G=\Theta(l_1,l_2,l_3)$, where $l_1$, $l_2$, and $l_3$ all have the same parity and $\{l_1,l_2,l_3\}\neq\{2\}$. If $H=K_1\vee G$, then $H$ is enumeratively chromatic-choosable.
\end{theorem}

By the results of Section~\ref{almost}, we have $P_{\ell}(G,m)=P(G,m)$ whenever $m\geq3$. Thus, by Theorem~\ref{thm:examples}, $P_{\ell}(H,m)=P(H,m)$ whenever $m\geq4$. Since $G$ is connected and bipartite, $P(H,3)=6$. Therefore, to prove Theorem~\ref{thm: conetheta}, it remains only to show that $P_{\ell}(H,3)=6$. To do this, we will use the Alon--Tarsi Theorem and an enumerative version of the Combinatorial Nullstellensatz.

\begin{theorem}[\cite{AT92}] \label{thm: AandT}
    Let $D$ be an orientation of the simple graph $G$. Suppose that $L$ is a list assignment for $G$ such that $|L(y)| \geq d^+_D(y) + 1$ for every $y \in V(G)$. If the numbers of even and odd circulations contained in $D$ differ, then $G$ is $L$-colorable.
\end{theorem}

\begin{theorem} [\cite{AF93}] \label{thm: AandF} 
Let $\mathbb{F}$ be an arbitrary field, let $A_1$, $A_2$, $\ldots$, $A_n$ be any non-empty subsets of $\mathbb{F}$, and let $B = \prod_{i=1}^n A_i$.  Suppose that $P \in \mathbb{F}[c_1, \ldots, c_n]$ is a polynomial of degree $d$ that does not vanish on all of $B$.  Then, the number of points in $B$ for which $P$ has a non-zero value is at least $\min \prod_{i=1}^n q_i$ where the minimum is taken over all integers $q_i$ such that $1 \leq q_i \leq |A_i|$ and $\sum_{i=1}^n q_i \geq -d + \sum_{i=1}^n |A_i|$.
\end{theorem}

We will use Theorem~\ref{thm: AandT} to establish the following proposition.

\begin{pro} \label{pro: one3list}
Suppose $G=\Theta(l_1,l_2,l_3)$, where $l_1$, $l_2$, and $l_3$ all have the same parity. Suppose $L$ is a list assignment for $G$ such that $|L(a)|=3$ for exactly one vertex $a\in V(G)$ and $|L(b)|=2$ for every $b\in V(G)-\{a\}$. Then $G$ is $L$-colorable.
\end{pro}

\begin{proof}
First, suppose $a\in V(G)-\{x,y\}$. Without loss of generality, suppose $a\in V(S_1)$. Let $G'=S_2\cup S_3$, and let $L'$ be the restriction of $L$ to $V(G')$. Since $l_2$ and $l_3$ have the same parity, $G'$ is an even cycle, and hence $G'$ is $L'$-colorable. Let $h$ be a proper $L'$-coloring of $G'$. Starting from $x$ and $y$, we may greedily color the vertices of $S_1$ toward $a$. Since every vertex of $S_1$ other than $a$ has a list of size $2$ and $|L(a)|=3$, we can then color $a$ with a color different from the colors of its neighbors. Thus, there is a proper $L_1$-coloring of $S_1$ that agrees with $h$ on $x$ and $y$, and together these colorings form a proper $L$-coloring of $G$.

Now, suppose $a\in\{x,y\}$. By symmetry, we may assume $a=x$. Let $D$ be the orientation of $G$ obtained by orienting the edges of $S_1$ and $S_2$ from $x$ to $y$ and the edges of $S_3$ from $y$ to $x$. Then $d_D^+(x)=2$ and $d_D^+(b)=1$ for each $b\in V(G)-\{x\}$. Thus, $|L(v)|\geq d_D^+(v)+1$ for every $v\in V(G)$. Since $G$ is bipartite, $D$ has no odd circulations.  $D$ also clearly has at least one even circulation. Therefore, by Theorem~\ref{thm: AandT}, $G$ is $L$-colorable.
\end{proof}

Now, let $L$ be an arbitrary $3$-assignment for $H$. By relabeling colors if necessary, we may assume $L(w)=[3]$. For each $i\in[3]$, let $h(i)=|\{v\in V(G):i\notin L(v)\}|$. Without loss of generality, suppose $h(1)\geq h(2)\geq h(3)$. If $h(1)+h(2)+h(3)=0$, then $L(v)=[3]$ for every $v\in V(H)$, and hence $P(H,L)=P(H,3)=6$. Thus, we may assume $h(1)+h(2)+h(3)\geq1$. We complete the proof by showing $P(H,L) \geq 6$ in each of the following cases: $h(3)\geq2$, $h(3)=1$, and $h(3)=0$.

\begin{lemma} \label{lem: twocase}
If $h(3) \ge 2$, then $P(H, L) \ge 6$.
\end{lemma}

\begin{proof}
Since $h(1)\geq h(2)\geq h(3)\geq2$, we have $h(i)\geq2$ for each $i\in[3]$. For each $i\in[3]$, let $\mathcal{C}_i$ be the set of proper $L$-colorings of $H$ that color $w$ with $i$, and let $L_i(v)=L(v)-\{i\}$ for each $v\in V(G)$. Since $L(w)=[3]$, we have $|\mathcal{C}_i|=P(G,L_i)$ for each $i\in[3]$. Therefore,
$$P(H,L)=\sum_{i=1}^3|\mathcal{C}_i|=\sum_{i=1}^3P(G,L_i).$$
It suffices to show that $P(G,L_i)\geq2$ for each $i\in[3]$.

Fix a $j\in[3]$. We will show $P(G,L_j)\geq2$. Arbitrarily name the vertices of $G$ so that $V(G)=\{v_1,\ldots,v_n\}$. Note that $n=l_1+l_2+l_3-1$. Let $F\in\mathbb{R}[a_1,\ldots,a_n]$ be given by
$$F(a_1,\ldots,a_n)=\prod_{\substack{s>q,\\ v_qv_s\in E(G)}}(a_s-a_q).$$
Observe that $F$ is homogeneous of degree $n+1$. Suppose $B=\prod_{i=1}^nL_j(v_i)$. Notice that every proper $L_j$-coloring of $G$ corresponds to an element of $B$ where $F$ is nonzero. Since $h(j)\geq2$, at least two of the lists corresponding to $L_j$ have size $3$, and $|L_j(v)|\geq2$ for each $v\in V(G)$. By restricting all but one of the lists of size $3$ to lists of size $2$, Proposition~\ref{pro: one3list} implies that $F$ does not vanish on all of $B$.

By Theorem~\ref{thm: AandF}, the number of points in $B$ for which $F$ has a nonzero value is at least $\min\prod_{i=1}^n q_i$, where the minimum is taken over all integers $q_i$ such that $1\leq q_i\leq |L_j(v_i)|$ for each $i\in[n]$ and
$$\sum_{i=1}^n q_i\geq -(n+1)+\sum_{i=1}^n|L_j(v_i)|\geq n+1.$$
It is easy to see from these conditions that $\min\prod_{i=1}^n q_i\geq2$. Therefore, $P(G,L_j)\geq2$. Since $j\in[3]$ was arbitrary, $P(G,L_i)\geq2$ for each $i\in[3]$.  Thus, $P(H,L)\geq6$.
\end{proof}

We now turn to the cases $h(3)=1$ and $h(3)=0$. We will use the following simple observation about list colorings of stars.

\begin{obs} \label{obs: star}
Suppose $G=K_{1,n}$ with bipartition $\{a\},\{b_1,\ldots,b_n\}$. If $L$ is a list assignment for $G$ with $c\in L(a)$, then the number of proper $L$-colorings that color $a$ with $c$ is $\prod_{i=1}^n |L(b_i)-\{c\}|.$
\end{obs}

\begin{lemma} \label{lem: onecase}
If $h(3)=1$, then $P(H,L)\geq6$.
\end{lemma}

\begin{proof}
We first suppose that $3\in L(v)$ for each $v\in X$, and let $b\in Y$ be the unique vertex such that $3\notin L(b)$. The proof when the unique vertex whose list does not contain $3$ belongs to $X$ is analogous.

First, suppose $\{1,2\}\subseteq L(y)$ for each $y\in Y$. For each $i\in[2]$, let $\mathcal{C}_{(Y,i)}$ be the set of proper $L$-colorings of $H$ that color each vertex in $Y$ with $i$. Since $\mathcal{C}_{(Y,1)}$ and $\mathcal{C}_{(Y,2)}$ are disjoint, $P(H,L)\geq|\mathcal{C}_{(Y,1)}|+|\mathcal{C}_{(Y,2)}|$.

Suppose each vertex in $Y$ is colored with $2$. Let $L'(v)=L(v)-\{2\}$ for each $v\in V(H-Y)$. Notice that a proper $L'$-coloring of $H-Y$ completes a proper $L$-coloring of $H$. Since $h(2)\geq h(3)=1$ and $\{1,2\}\subseteq L(y)$ for each $y\in Y$, there is a $z\in X$ such that $|L'(z)|=3$. Since $H-Y=K_{1,|X|}$, Observation~\ref{obs: star} implies there are at least $2$ proper $L'$-colorings of $H-Y$ that color $w$ with $1$ and at least $2$ proper $L'$-colorings of $H-Y$ that color $w$ with $3$. Consequently, $|\mathcal{C}_{(Y,2)}|\geq4$. A similar argument shows that $|\mathcal{C}_{(Y,1)}|\geq4$. Thus, $P(H,L)\geq6$.

Now, suppose there is a $z\in Y$ such that $\{1,2\}\not\subseteq L(z)$; note that it is possible that $z=b$. Let $\mathcal{C}_{(X,3)}$ be the set of proper $L$-colorings of $H$ that color each vertex in $X$ with $3$, and let $\mathcal{C}_{(w,3)}$ be the set of proper $L$-colorings of $H$ that color $w$ with $3$. Since $\mathcal{C}_{(X,3)}$ and $\mathcal{C}_{(w,3)}$ are disjoint, $P(H,L)\geq|\mathcal{C}_{(X,3)}|+|\mathcal{C}_{(w,3)}|$.

If we color $w$ with $3$ and let $L'(v)=L(v)-\{3\}$ for each $v\in V(G)$, then exactly one list corresponding to $L'$ has size $3$, while all remaining lists have size $2$. Therefore, Proposition~\ref{pro: one3list} implies that there is at least one way to complete a proper $L$-coloring of $H$ when $w$ is colored with $3$. Hence, $|\mathcal{C}_{(w,3)}|\geq1$. It remains to show that $|\mathcal{C}_{(X,3)}|\geq5$.

Suppose each vertex in $X$ is colored with $3$, and let $L'(v)=L(v)-\{3\}$ for each $v\in V(H-X)$. Notice that a proper $L'$-coloring of $H-X$ completes a proper $L$-coloring of $H$. If $z=b$, consider $|L'(b)-\{1\}|$ and $|L'(b)-\{2\}|$. Both numbers are at least $2$, and at least one is $3$. Since $H-X=K_{1,|Y|}$, Observation~\ref{obs: star} implies there are at least $|L'(b)-\{1\}|$ proper $L'$-colorings of $H-X$ that color $w$ with $1$ and at least $|L'(b)-\{2\}|$ proper $L'$-colorings of $H-X$ that color $w$ with $2$. Consequently, $|\mathcal{C}_{(X,3)}|\geq5$.

If $z\neq b$, consider $|L'(b)-\{j\}|$ for each $j\in[2]$. Both of these numbers are at least $2$. Also, $|L'(z)-\{1\}|$ and $|L'(z)-\{2\}|$ are both at least $1$, and at least one is $2$. Since $H-X=K_{1,|Y|}$, Observation~\ref{obs: star} implies there are at least $(|L'(b)-\{1\}|)(|L'(z)-\{1\}|)$ proper $L'$-colorings of $H-X$ that color $w$ with $1$ and at least $(|L'(b)-\{2\}|)(|L'(z)-\{2\}|)$ proper $L'$-colorings of $H-X$ that color $w$ with $2$. Consequently,
$$|\mathcal{C}_{(X,3)}|\geq(2)(2)+(2)(1)=6.$$
Thus, $P(H,L)\geq6$.
\end{proof}

\begin{lemma} \label{lem: zerocase}
If $h(3)=0$, then $P(H,L)\geq6$.
\end{lemma}

\begin{proof}
Since $h(3)=0$, $3\in L(v)$ for each $v\in V(H)$. We may suppose $h(1)\geq1$, since if $h(1)=0$, then $L$ assigns $[3]$ to every vertex of $H$ and $P(H,L)=P(H,3)=6$.

We first show that $P(H,L)\geq6$ in the case that $h(2)=0$. Let $\mathcal{C}_{(w,j)}$ be the set of all proper $L$-colorings of $H$ that color $w$ with $j$ for each $j\in[3]$. Since $\mathcal{C}_{(w,1)}$, $\mathcal{C}_{(w,2)}$, and $\mathcal{C}_{(w,3)}$ are pairwise disjoint, $P(H,L)\geq\sum_{i=1}^3|\mathcal{C}_{(w,i)}|$. We can construct a proper $L$-coloring as follows: first color $w$ with $2$, then color each vertex in $X$ (resp. $Y$) with $3$, and finally color each $v\in Y$ (resp. $v\in X$) with an element of $L(v)-\{2,3\}$. Thus, $|\mathcal{C}_{(w,2)}|\geq2$. A similar argument shows that $|\mathcal{C}_{(w,3)}|\geq2$. Finally, a proper $L$-coloring can be obtained by coloring $w$ with $1$, every vertex in $X$ (resp. $Y$) with $2$, and every vertex in $Y$ (resp. $X$) with $3$. Thus, $|\mathcal{C}_{(w,1)}|\geq2$. Consequently, $P(H,L)\geq6$.

We will now show that $P(H,L)\geq6$ in the case $h(2)\geq1$. Suppose $q\in V(H-w)$ is such that $2\notin L(q)$. Also, suppose $z\in V(H-w)$ is such that $1\notin L(z)$; note that it could be that $z=q$. Let $\mathcal{C}_{(X,3)}$ be the set of proper $L$-colorings of $H$ that color every vertex in $X$ with $3$, and let $\mathcal{C}_{(Y,3)}$ be the set of proper $L$-colorings of $H$ that color every vertex in $Y$ with $3$. Since $\mathcal{C}_{(X,3)}$ and $\mathcal{C}_{(Y,3)}$ are disjoint, $P(H,L)\geq|\mathcal{C}_{(X,3)}|+|\mathcal{C}_{(Y,3)}|$.

First, suppose $z=q$, and suppose $z\in X$; the case that $z\in Y$ is similar. Color each vertex in $Y$ with $3$, and let $L'(v)=L(v)-\{3\}$ for each $v\in X\cup\{w\}$. Notice that a proper $L'$-coloring of $H-Y$ completes a proper $L$-coloring of $H$. If $w$ is colored with $j$ for some $j\in[2]$, then since $H-Y=K_{1,|X|}$, Observation~\ref{obs: star} implies there are at least $|L'(z)-\{j\}|$ ways to complete a proper $L$-coloring of $H$ that colors $w$ with $j$ and every vertex in $Y$ with $3$. Since $|L'(z)-\{1\}|\geq2$ and $|L'(z)-\{2\}|\geq2$, we have $|\mathcal{C}_{(Y,3)}|\geq4$.

Now, color each vertex in $X$ with $3$, and let $L''(v)=L(v)-\{3\}$ for each $v\in Y\cup\{w\}$. Note that $L''$ is a $2$-assignment for $H-X$. Since $H-X$ is a star and hence chordal, Theorem~\ref{thm:examples} implies that $H-X$ is enumeratively chromatic-choosable. Thus, there are at least $2$ proper $L''$-colorings of $H-X$. Notice that a proper $L''$-coloring of $H-X$ completes a proper $L$-coloring of $H$. Thus, $|\mathcal{C}_{(X,3)}|\geq2$, and consequently $P(H,L)\geq6$.

Now, suppose $z\neq q$. First, consider the case where $z,q\in X$; the case that $z,q\in Y$ is similar. Color each vertex in $Y$ with $3$, and let $L'(v)=L(v)-\{3\}$ for each $v\in X\cup\{w\}$. Notice that a proper $L'$-coloring of $H-Y$ completes a proper $L$-coloring of $H$ in which every vertex in $Y$ is colored with $3$. Since $1\notin L'(z)$, Observation~\ref{obs: star} implies there are at least $2$ such colorings when $w$ is colored with $1$. Similarly, since $2\notin L'(q)$, there are at least $2$ such colorings when $w$ is colored with $2$. Thus, $|\mathcal{C}_{(Y,3)}|\geq4$.

Now, color each vertex in $X$ with $3$, and let $L''(v)=L(v)-\{3\}$ for each $v\in Y\cup\{w\}$. Note that $L''$ is a $2$-assignment for $H-X$. Since $H-X$ is a star and hence chordal, Theorem~\ref{thm:examples} implies that $H-X$ is enumeratively chromatic-choosable. Thus, there are at least $2$ proper $L''$-colorings of $H-X$. Notice that a proper $L''$-coloring of $H-X$ completes a proper $L$-coloring of $H$ in which every vertex in $X$ is colored with $3$. Consequently, $P(H,L)\geq|\mathcal{C}_{(X,3)}|+|\mathcal{C}_{(Y,3)}|\geq6$.

Finally, consider the case when $z\in X$ and $q\in Y$; the case that $z\in Y$ and $q\in X$ is similar. Color each vertex in $X$ with $3$. Let $L'(v)=L(v)-\{3\}$ for each $v\in Y\cup\{w\}$. Clearly, $|L'(v)|=2$ for each $v\in Y\cup\{w\}$. Notice that a proper $L'$-coloring of $H-X$ completes a proper $L$-coloring of $H$ in which every vertex in $X$ is colored with $3$. Since $H-X=K_{1,|Y|}$ and $2\notin L'(q)$, Observation~\ref{obs: star} implies there are at least $2$ proper $L'$-colorings of $H-X$ when $w$ is colored with $2$ and at least $1$ proper $L'$-coloring when $w$ is colored with $1$. Thus, $|\mathcal{C}_{(X,3)}|\geq3$.

Now, color each vertex in $Y$ with $3$. Let $L''(v)=L(v)-\{3\}$ for each $v\in X\cup\{w\}$. Clearly, $|L''(v)|=2$ for each $v\in X\cup\{w\}$. Notice that a proper $L''$-coloring of $H-Y$ completes a proper $L$-coloring of $H$ in which every vertex in $Y$ is colored with $3$. Since $H-Y=K_{1,|X|}$ and $1\notin L''(z)$, Observation~\ref{obs: star} implies there are at least $2$ proper $L''$-colorings of $H-Y$ when $w$ is colored with $1$ and at least $1$ proper $L''$-coloring when $w$ is colored with $2$. Thus, $|\mathcal{C}_{(Y,3)}|\geq3$. Consequently, $P(H,L)\geq|\mathcal{C}_{(X,3)}|+|\mathcal{C}_{(Y,3)}|\geq6.$
\end{proof}

Lemmas~\ref{lem: twocase}, \ref{lem: onecase}, and~\ref{lem: zerocase} show that $P(H,L)\geq6$ for every $3$-assignment $L$ of $H$. Hence, $P_{\ell}(H,3)\geq6$, and we have $P_{\ell}(H,3)=P(H,3)$. Together with the discussion preceding Theorem~\ref{thm: conetheta}, this proves Theorem~\ref{thm: conetheta}, and hence Theorem~\ref{thm: cone}.

\paragraph{Declaration of generative AI and AI-assisted technologies in
the manuscript preparation process.}
During the preparation of this manuscript, the authors used ChatGPT
(OpenAI) for limited language editing and manuscript checking. The authors take full responsibility for the final content of the manuscript.

\paragraph{Acknowledgment.} The authors would like to thank Hemanshu Kaul for helpful conversations.  This project was completed by the Coloring Research Group of South Alabama at the University of South Alabama during the spring 2026, summer 2026, and fall 2026 semesters.  The support of the University of South Alabama is gratefully acknowledged.

\bibliographystyle{hplain}
\bibliography{thetabib}{}

@article{AF93,
    author = {N. Alon and Z. F\"{u}redi},
    title = {Covering the cube by affine hyperplanes},
    journal = {European J. Combin.},
    volume = {14},
    year = {1993},
    pages = {79--83}
}

@article{AM25,
  author  = {S. Allred and J. A. Mudrock},
  title   = {Enumerative chromatic choosability},
  journal = {Electron. J. Combin.},
  volume  = {33},
  number  = {4},
  year    = {2026},
  pages   = {Paper No. P4.6},
  doi     = {10.37236/14176},
  url     = {https://doi.org/10.37236/14176}
}

@article{AS90,
  author = {Kostochka, A. V. and Sidorenko, A. F.},
  title = {Problems Proposed at the Problem Session of the Prachatice Conference on Graph Theory},
  journal = {Ann. Discrete Math.},
  volume = {51},
  pages = {375--384},
  year = {1992},
  doi = {10.1016/S0167-5060(08)70659-9},
  url = {https://doi.org/10.1016/S0167-5060(08)70659-9}
}

@article{AT92,
  author  = {N. Alon and M. Tarsi},
  title   = {Colorings and Orientations of Graphs},
  journal = {Combinatorica},
  volume  = {12},
  number  = {2},
  year    = {1992},
  pages   = {125--134},
  doi     = {10.1007/BF01204715}
}

@book {B94,
    AUTHOR = {Biggs, N.},
     TITLE = {Algebraic graph theory},
    SERIES = {Cambridge Mathematical Library},
   EDITION = {Second},
 PUBLISHER = {Cambridge University Press, Cambridge},
      YEAR = {1993},
     PAGES = {viii+205},
      ISBN = {0-521-45897-8},
   MRCLASS = {05C50},
  MRNUMBER = {1271140},
MRREVIEWER = {Robin\ J.\ Wilson},
}

@article {B12,
    AUTHOR = {Birkhoff, G. D.},
     TITLE = {A determinant formula for the number of ways of coloring a
              map},
   JOURNAL = {Ann. of Math. (2)},
  FJOURNAL = {Annals of Mathematics. Second Series},
    VOLUME = {14},
      YEAR = {1912},
    NUMBER = {1-4},
     PAGES = {42--46},
      ISSN = {0003-486X,1939-8980},
   MRCLASS = {99-04},
  MRNUMBER = {1502436},
       DOI = {10.2307/1967597},
       URL = {https://doi.org/10.2307/1967597},
}

@article {BK23,
    AUTHOR = {Bui, M. V. and Kaul, H. and Maxfield, M. and
              Mudrock, J. A. and Shin, P. and Thomason, S.},
     TITLE = {Non-chromatic-adherence of the {DP} color function via
              generalized theta graphs},
   JOURNAL = {Graphs Combin.},
  FJOURNAL = {Graphs and Combinatorics},
    VOLUME = {39},
      YEAR = {2023},
    NUMBER = {3},
     PAGES = {Paper No. 42},
      ISSN = {0911-0119,1435-5914},
   MRCLASS = {05C15 (05A20 05C30 05C69)},
  MRNUMBER = {4570386},
       DOI = {10.1007/s00373-023-02633-z},
       URL = {https://doi.org/10.1007/s00373-023-02633-z},
}

@article {CL26,
    AUTHOR = {Chi, Y. and Lee, S. and Morrissette, F. and Mudrock, J. A. and Nguyen, G. and Whatley, B.},
    TITLE = {Enumeratively chromatic-choosable theta graphs},
    JOURNAL = {Enumerative Combinatorics and Applications},
    VOLUME = {6},
    YEAR = {2026},
    NUMBER = {3},
    PAGES = {Article \#S2R23},
    DOI = {10.54550/ECA2026V6S3R23},
    URL = {https://doi.org/10.54550/ECA2026V6S3R23},
}

@article {DZ22,
    AUTHOR = {Dong, F. and Zhang, M.},
     TITLE = {An improved lower bound of {$P(G, L)- P(G, k)$} for
              {$k$}-assignments {$L$}},
   JOURNAL = {J. Combin. Theory Ser. B},
  FJOURNAL = {Journal of Combinatorial Theory. Series B},
    VOLUME = {161},
      YEAR = {2023},
     PAGES = {109--119},
      ISSN = {0095-8956,1096-0902},
   MRCLASS = {05C31 (05C15)},
  MRNUMBER = {4553004},
       DOI = {10.1016/j.jctb.2023.02.002},
       URL = {https://doi.org/10.1016/j.jctb.2023.02.002},
}

@article {D92,
    AUTHOR = {Donner, Q.},
     TITLE = {On the number of list-colorings},
   JOURNAL = {J. Graph Theory},
  FJOURNAL = {Journal of Graph Theory},
    VOLUME = {16},
      YEAR = {1992},
    NUMBER = {3},
     PAGES = {239--245},
      ISSN = {0364-9024,1097-0118},
   MRCLASS = {05C15},
  MRNUMBER = {1168583},
MRREVIEWER = {N.\ Z.\ Li},
       DOI = {10.1002/jgt.3190160307},
       URL = {https://doi.org/10.1002/jgt.3190160307},
}

@article{DP15,
    TITLE = {Correspondence coloring and its application to list-coloring planar graphs without cycles of lengths 4 to 8},
    JOURNAL = {J. Combin. Theory Ser. B},
    FJOURNAL = {Journal of Combinatorial Theory, Series B},
    VOLUME = {129},
    PAGES = {38--54},
    YEAR = {2018},
    ISSN = {0095-8956},
    DOI = {10.1016/j.jctb.2017.09.001},
    URL = {https://doi.org/10.1016/j.jctb.2017.09.001},
    AUTHOR = {Z. Dvo\v{r}\'{a}k and L. Postle}
}

@inproceedings {ET79,
    AUTHOR = {Erd\H{o}s, P. and Rubin, A. L. and Taylor, H.},
     TITLE = {Choosability in graphs},
 BOOKTITLE = {Proceedings of the {W}est {C}oast {C}onference on
              {C}ombinatorics, {G}raph {T}heory and {C}omputing ({H}umboldt
              {S}tate {U}niv., {A}rcata, {C}alif., 1979)},
    SERIES = {Congress. Numer.},
    VOLUME = {XXVI},
     PAGES = {125--157},
 PUBLISHER = {Utilitas Math., Winnipeg, MB},
      YEAR = {1980},
      ISBN = {0-919628-26-5},
   MRCLASS = {05C15},
  MRNUMBER = {593902},
MRREVIEWER = {R.\ C.\ Entringer},
}

@misc{DZ26,
  author    = {M. Zhang and F. Dong},
  title     = {Non-persistence of equality between chromatic polynomials
and list-color functions},
  howpublished   = {arXiv:2608.19773},
  year      = {2026},
  note      = {preprint}
}

@article {HC92,
    AUTHOR = {H\"{a}ggkvist, R. and Chetwynd, A.},
     TITLE = {Some upper bounds on the total and list chromatic numbers of
              multigraphs},
   JOURNAL = {J. Graph Theory},
  FJOURNAL = {Journal of Graph Theory},
    VOLUME = {16},
      YEAR = {1992},
    NUMBER = {5},
     PAGES = {503--516},
      ISSN = {0364-9024,1097-0118},
   MRCLASS = {05C15},
  MRNUMBER = {1185013},
MRREVIEWER = {Hugh\ Hind},
       DOI = {10.1002/jgt.3190160510},
       URL = {https://doi.org/10.1002/jgt.3190160510},
}

@article {KK23,
    AUTHOR = {Kaul, H. and Kumar, A. and Liu, A. and Mudrock,
              J. A. and Rewers, P. and Shin, P. and Tanahara,
              M. S. and To, K.},
     TITLE = {Bounding the list color function threshold from above},
   JOURNAL = {Involve},
  FJOURNAL = {Involve. A Journal of Mathematics},
    VOLUME = {16},
      YEAR = {2023},
    NUMBER = {5},
     PAGES = {849--882},
      ISSN = {1944-4176,1944-4184},
   MRCLASS = {05C15 (05C30)},
  MRNUMBER = {4677029},
       DOI = {10.2140/involve.2023.16.849},
       URL = {https://doi.org/10.2140/involve.2023.16.849},
}

@article{KM18,
  author = {H. Kaul and J. A. Mudrock},
  title = {Criticality, the list color function, and list coloring the {C}artesian product of graphs},
  journal = {J. Comb.},
  fjournal = {Journal of Combinatorics},
  volume = {12},
  year = {2021},
  number = {3},
  pages = {479--514},
  doi = {10.4310/joc.2021.v12.n3.a4},
  url = {https://doi.org/10.4310/joc.2021.v12.n3.a4}
}

@article{KM19,
  author = {H. Kaul and J. A. Mudrock},
  title = {On the chromatic polynomial and counting {DP}-colorings of graphs},
  journal = {Adv. in Appl. Math.},
  volume = {123},
  year = {2021},
  pages = {102131},
  doi = {10.1016/j.aam.2020.102131},
  url = {https://doi.org/10.1016/j.aam.2020.102131}
}

@article {KN16,
    AUTHOR = {Kirov, R. and Naimi, R.},
     TITLE = {List coloring and {$n$}-monophilic graphs},
   JOURNAL = {Ars Combin.},
  FJOURNAL = {Ars Combinatoria. A Canadian Journal of Combinatorics},
    VOLUME = {124},
      YEAR = {2016},
     PAGES = {329--340},
      ISSN = {0381-7032,2817-5204},
   MRCLASS = {05C15},
  MRNUMBER = {3468819},
}

@book {MOA11,
    AUTHOR = {Marshall, A. W. and Olkin, I. and Arnold, B. C.},
     TITLE = {Inequalities: theory of majorization and its applications},
    SERIES = {Springer Series in Statistics},
   EDITION = {Second},
 PUBLISHER = {Springer, New York},
      YEAR = {2011},
     PAGES = {xxvii+909},
      ISBN = {978-0-387-40087-7},
}

@article {O02,
    AUTHOR = {Ohba, K.},
     TITLE = {On chromatic-choosable graphs},
   JOURNAL = {J. Graph Theory},
  FJOURNAL = {Journal of Graph Theory},
    VOLUME = {40},
      YEAR = {2002},
    NUMBER = {2},
     PAGES = {130--135},
      ISSN = {0364-9024,1097-0118},
   MRCLASS = {05C15},
  MRNUMBER = {1899118},
       DOI = {10.1002/jgt.10033},
       URL = {https://doi.org/10.1002/jgt.10033},
}

@article {T09,
    AUTHOR = {Thomassen, C.},
     TITLE = {The chromatic polynomial and list colorings},
   JOURNAL = {J. Combin. Theory Ser. B},
  FJOURNAL = {Journal of Combinatorial Theory. Series B},
    VOLUME = {99},
      YEAR = {2009},
    NUMBER = {2},
     PAGES = {474--479},
      ISSN = {0095-8956,1096-0902},
   MRCLASS = {05C15},
  MRNUMBER = {2482965},
MRREVIEWER = {Gek\ Ling\ Chia},
       DOI = {10.1016/j.jctb.2008.09.005},
       URL = {https://doi.org/10.1016/j.jctb.2008.09.005},
}

@article {V76,
    AUTHOR = {Vizing, V. G.},
    TITLE = {Coloring the vertices of a graph in prescribed colors},
    JOURNAL = {Diskret. Analiz},
    YEAR = {1976},
    NUMBER = {29},
    PAGES = {3--10},
    NOTE = {In Russian},
    MRCLASS = {05C15},
    MRNUMBER = {498216},
    MRREVIEWER = {D.\ Cvetkovi\'c},
}

@article{WQ17,
  author = {W. Wang and J. Qian and Z. Yan},
  title = {When does the list-coloring function of a graph equal its chromatic polynomial},
  journal = {J. Combin. Theory Ser. B},
  volume = {122},
  year = {2017},
  pages = {543--549},
  doi = {10.1016/j.jctb.2016.08.002},
  url = {https://doi.org/10.1016/j.jctb.2016.08.002}
}

@book {W01,
    AUTHOR = {West, D. B.},
    TITLE = {Introduction to graph theory},
    EDITION = {Second},
    PUBLISHER = {Prentice Hall, Inc., Upper Saddle River, NJ},
    YEAR = {2001},
    ISBN = {0-13-014400-2},
}
\vspace{-5mm}

\end{document}